\documentclass{article}
\usepackage{amsfonts}
\usepackage{amsmath}
\usepackage{amssymb}
\usepackage{amsthm}
\usepackage{graphicx}
\usepackage{braket}
\usepackage{typearea}
\usepackage{fullpage}

\def\eqbd{\mathop{{:}{=}}}

\theoremstyle{definition}
\newtheorem{definition}{Definition}
\newtheorem{lemma}[definition]{Lemma}
\newtheorem{theorem}[definition]{Theorem}
\newtheorem{proposition}[definition]{Proposition}
\newtheorem{corollary}[definition]{Corollary}

\newenvironment{@abssec}[1]{%
     \if@twocolumn
       \section*{#1}%
     \else
       \vspace{.05in}\footnotesize
       \parindent .2in
         {\bfseries #1. }\ignorespaces
     \fi}
     {\if@twocolumn\else\par\vspace{.1in}\fi}
\newenvironment{keywords}{\begin{@abssec}{Key words}}{\end{@abssec}}
\newenvironment{AMS}{\begin{@abssec}{AMS subject classification}}{\end{@abssec}}

\newcommand{\ang}[1]{\langle\! \langle#1\rangle\! \rangle}
\DeclareMathOperator{\tr}{tr}

\begin{document}

\title{Szeg\H{o}'s theorem for matrix orthogonal polynomials}


\author{Maxim Derevyagin\footnotemark[1], Olga Holtz
\footnotemark[1] \footnotemark[2] ,
Sergey Khrushchev\footnotemark[3], and Mikhail 
Tyaglov\footnotemark[1]}

\renewcommand{\thefootnote}{\fnsymbol{footnote}}

\footnotetext[1]{
Department of Mathematics MA 4-5,
Technische Universit\"at Berlin,
Strasse des 17. Juni 136,
D-10623 Berlin,
Germany.}
\footnotetext[2]{
Department of Mathematics,
  University of California-Berkeley,
      821 Evans Hall,
      Berkeley, California, 94720.
Telephone: +1 510 642 2122 
Fax: +1 510 642 8204
Email: holtz@math.berkeley.edu
}
\footnotetext[3]{
Department of Mathematics
Kazakh-British University
Tole bi 59, 050000
Almaty, Kazakhstan.
}



\date{\small December 9, 2011}
\maketitle



\begin{abstract}
We  extend some classical theorems in the theory of orthogonal polynomials on the unit circle to the matrix case.
In particular, we prove a matrix analogue of Szeg\H{o}'s theorem.
As a by-product, we also obtain an elementary proof of
the distance formula by Helson and Lowdenslager.\end{abstract}
\maketitle

\section{Historical background and motivation}

The classical work~\cite{Szego} of Szeg\H{o} was the first to address the asymptotics
of orthogonal polynomials on the unit circle $\mathbb{T}$ under the assumption that
the entropy of the underlying measure $\sigma$ is finite, i.e.,
$$ \int_{\mathbb T} \log \sigma' {d\theta \over 2\pi} > - \infty. $$
Further aspects of Szeg\H{o}'s theory were developed by Geronimus, Verblunsky and 
 others, which led to a number of other formulas, in various setups, involving 
the entropy such as the formula of Helson-Lowdenslager \cite{HelsonLowdenslager} 
for multivariate random processes (for a historical account, see \cite[\S 1.1]{SimonI}).

Verblunsky \cite[formulas (v) and (vi)]{Verblunsky} showed that, for any probability measure
$\sigma$ on the unit circle ${\mathbb T}$,
\begin{equation}\label{GeronimSzego}
\lim_n\prod_{k=0}^n(1-|a_k|^2)=\exp\int_{\mathbb{T}}\log\sigma^{\,\prime}\,\frac{d\theta}{2\pi}.
\end{equation}
Here $\{a_k\}_{k\geq 0}$ is a sequence of points in the unit disc $\mathbb{D}$ called the \textit{parameters} of $\sigma$ \cite[\S 8.3]{Khrushchev} and $\sigma^{\,\prime}=2\pi d\sigma/d\theta$ is the Lebesgue derivative of $\sigma$. The numbers $\{a_k\}_{k\geq 0}$ have different names depending on the area where they are considered. In the theory of orthogonal polynomials they are known as the {\em Szeg\H{o} recurrence coefficients, Verblunsky parameters, Geronimus parameters,\/} in Schur's theory they are {\em Schur's parameters,\/} in inverse scattering problems they are {\em reflection coefficients,\/} see \cite[\S 1.1]{SimonI}.

In the matrix setting,  $\mathbf{\sigma}$ is a Borel measure on $\mathbb{T}$ with values in the
set ${\mathcal M}_\ell^{+}$ of all nonnegative definite matrices in ${\mathcal M}_\ell$,
the set of all $\ell\times \ell$ matrices with complex entries. We denote by 
$\textsf{P}_{\ell}(\mathbb{T})$ the set of all matrix-valued
nonnegative measures $\sigma$ on $\mathbb{T}$ that are normalized, i.e.,
$$\sigma(\mathbb{T})=\mathbf{1}$$ to the unit matrix $\mathbf{1}$ 
 in ${\mathcal M}_\ell^{+}$. We refer to $\textsf{P}_{\ell}(\mathbb{T})$ 
as the class of {\em matrix probability measures}.

The matrix case is important in multivariate Time Series and Prediction Theory \cite{Hannan,HelsonLowdenslager,Masani,Rozanov,Wiener}. As far as we know, 
the first Szeg\H{o}-type results  on matrix-valued orthogonal polynomials
were obtained by Delsarte, Genin and Kamp \cite{DGK:78}; this line of research
was continued by Aptekarev and Nikishin~\cite{AN:83}.

Our method is a combination of the recent theory of matrix orthogonal polynomials presented in \cite{DPS:08} and the approach to Szeg\H{o}'s theory developed in \cite{Khrushchev2001, Khrushchev}. This combination allows us to avoid using factorization theory in matrix Hardy classes. Instead, we use only methods of Real Analysis and Matrix/Operator Theory.

Our main goal is to respond to the following challenge
of Damanik, Pushnitski, and Simon~\cite{DPS:08}:

{\it Among the deepest and most elegant methods in OPUC 
are those of Khrushchev [125, 126, 101].
We have not been able to extend them to MOPUC! 
We regard their extension as an important
open question...}
 
Below we provide a full matrix-valued version of Szeg\H{o}'s theorem, yielding
the previously known trace versions  as corollaries of our matrix formula.

Throughout the paper, we mostly follow the notation and terminology of~\cite{DPS:08}.

\section{Main results}

In the matrix case, the parameters $\alpha_k$ are matrices in ${\mathcal M}_\ell$ with norms $\Vert{\alpha}_k\Vert$  not exceeding $1$. Here $\Vert{\alpha}\Vert$ is the norm of the linear operator  defined by the matrix $\alpha$ subordinate to the usual Euclidean vector norm ($2$-norm)  on ${\mathbb C}^\ell$. This operator norm is also known as the {\em spectral\/} or the {\em Euclidean\/} norm.
This norm is well known to equal the largest singular value of the matrix $\alpha$;
in particular, if $\alpha$ is self-adjoint, the norm $\Vert{\alpha}\Vert$ equals
the spectral radius of $\alpha$.

 We denote by
$\alpha^\dagger$ the Hermitian conjugate of $\alpha\in{\mathcal M}_\ell$. The symbol  $\,^*\,$ is reserved for the
Szeg\H{o} dual, so we do not use it for the adjoint (see \eqref{reverseP}).

We assume that  the matrix
\begin{equation}\label{matrixpoly}
\int_{\mathbb{T}}p(e^{i\theta})^\dagger d\sigma(\theta)p(e^{i\theta})
\end{equation}
is positive (definite) for any polynomial $p$ with coefficients in ${\mathcal M}_\ell$.
 Condition~\eqref{matrixpoly} is equivalent to the requirement that $d\sigma$ has full rank $\ell$ for infinitely many points in $\mathbb{T}$. Under this condition, the right (left) orthogonal matrix polynomials $\varphi_n^R$ ($\varphi_n^L$) are uniquely determined by the standard Gram-Schmidt orthonormalization. It is important to notice that the left orthogonal matrix polynomials are obtained with respect to the left quadratic `form':
\begin{equation}\label{matrixpoly2int}
\int_{\mathbb{T}}p(e^{i\theta}) d\sigma(\theta)p(e^{i\theta})^\dagger .
\end{equation}
 Every $\sigma\in\textsf{P}_{\ell}(\mathbb{T})$ is uniquely determined by the sequence of its parameters $\{\alpha_k\}_{k\geq 0}$. These parameters are contractive matrices in ${\mathcal M}_\ell$.  If $\sigma$ is a matrix-valued measure with parameters $\{\alpha_k\}_{k\geq 0}$, then the parameters $\{\alpha_k^\dagger\}_{k\geq 0}$ correspond to the measure $\overline{\sigma}$ such that
 \begin{equation*}
 \overline{\sigma}(E)=\sigma(\overline{E}),\quad \overline{E}=\{\overline{z}:z\in E\},
 \end{equation*}
for any Borel set $E$, where $\overline{z}$ stands for the complex conjugate of a complex number $z$. We write $\varphi_n(z,\sigma)$ for the
orthogonal polynomials  if the dependence on $\sigma$ is important.

For a matrix polynomial $P_n$ of degree $n$, we define the {\em reversed\/}
 (or {\em Szeg\H{o} dual)\/}  polynomial $P_n^*$ by
\begin{equation} \label{reverseP}
P_n^*(z) = z^n P_n(1/\bar{z})^\dagger.
\end{equation}
The relationship between the left orthogonal polynomials ${\varphi}_n^L$ and the right orthogonal polynomials ${\varphi}_n^R$ is given by the formula
\begin{equation}\label{lefttoright45}
{\varphi}_n^L(e^{i\theta},\overline{\sigma}){=}{\varphi}_n^R(e^{-i\theta},\sigma)^\dagger
\end{equation}
(see Corollary \ref{eqnumberfive8}).
The $n$th left normalized orthogonal polynomial $\varphi_n^L(z,\overline{\sigma})$  depends on the parameters $\alpha_0^\dagger$, $\alpha_1^\dagger$, $\ldots$, $\alpha_{n-1}^\dagger$. Hence, the $n$th right polynomial $\varphi_n^R(z,\sigma)$ can be obtained from the left $\varphi_n^L(z,{\sigma})$ by replacing each $\alpha_k^\dagger$  by $\alpha_k$, replacing $z\in\mathbb{T}$ by $\overline{z}$ and applying the conjugation $\dagger$.

The main result of this paper is the following theorem. \bigskip

\noindent 
{\bf Theorem.\/} 
 Every matrix probability measure $\sigma\in\textsf{P}_{\ell}(\mathbb{T})$ satisfies
the following matrix equality:
\begin{equation}\label{theMainFormula}
\lim_{n\rightarrow +\infty}\int_{0}^{2\pi}\log([\varphi_n^{R,*} (e^{i\theta})^\dagger \varphi_n^{R,*} (e^{i
\theta})]^{-1})\frac{d\theta}{2\pi}=\int_{\mathbb{T}}\log\sigma^{\,\prime}\,\frac{d\theta}{2\pi}~.
\end{equation} \bigskip

\noindent
 If the parameters $\alpha_k$ of $\sigma$ form a family of commuting normal matrices, then
\eqref{theMainFormula} can be simplified to
 \begin{equation}\label{commutcase}
 \prod_{k=0}^{\infty}\left({\bf 1} -\alpha_k \alpha_k^\dagger\right)=\exp\int_{\mathbb{T}}\log\sigma^{\,\prime}\,\frac{d\theta}{2\pi}.
 \end{equation}

 \vskip 1mm\noindent
{\bf Remark.} Alternatively, the commuting case reduces to the diagonal and hence to the scalar case.
\vskip 2mm

 Regardless of the normality or commutativity of $\{ \alpha_k \}$, the following determinantal-trace version \cite{DGK:78} follows from~(\ref{theMainFormula}):
 \begin{equation}\label{traceclass}
 \prod_{k=0}^{\infty}\det\left({\bf 1} -\alpha_k \alpha_k^\dagger\right)=\exp\int_{\mathbb{T}}\text{tr}\log\sigma^{\,\prime}\,\frac{d\theta}{2\pi}.
 \end{equation}
The symmetry $z\mapsto \overline{z}$ keeps the Lebesgue measure on $\mathbb{T}$ invariant.
Hence, combining \eqref{lefttoright45} and a simple formula
 \begin{equation*}
\int_{\mathbb{T}}\log\overline{\sigma}^{\,\prime}\,\frac{d\theta}{2\pi}=\int_{\mathbb{T}}\log\sigma^{\,\prime}\,\frac{d\theta}{2\pi},
\end{equation*}
 we also obtain the left version of~(\ref{theMainFormula}):
\begin{equation}\label{LeftFormula4}
\lim_{n\rightarrow +\infty}\int_{0}^{2\pi}\log([\varphi_n^{L,*} (e^{i\theta}) \varphi_n^{L,*} (e^{i
\theta})^\dagger]^{-1})\frac{d\theta}{2\pi}=\int_{\mathbb{T}}\log\sigma^{\,\prime}\,\frac{d\theta}{2\pi}.
\end{equation}

\section{Matrix preliminaries}

Recall that we denote by ${\mathcal M}_\ell$ the ring of all $\ell\times \ell$ complex-valued
matrices, its identity matrix by ${\bf 1}$ and its zero matrix by ${\bf 0}$.
Along with the Euclidean norm $\| \cdot \|$ on ${\mathcal M}_\ell$,
we also consider the trace norm $\pmb{|\!\!|}\alpha\pmb{|\!\!|}_1=\text{tr}(\alpha^\dagger\alpha)^{1/2}$ and the Hilbert-Schmidt norm $\pmb{|\!\!|}\alpha\pmb{|\!\!|}_2=(\text{tr}(\alpha^\dagger\alpha))^{1/2}$. It is easy to see that
\begin{equation}\label{ineqfornorms}
\Vert{\alpha}\Vert\leq\pmb{|\!\!|}\alpha\pmb{|\!\!|}_2\leq \pmb{|\!\!|}\alpha\pmb{|\!\!|}_1\leq \ell\Vert{\alpha}\Vert.
\end{equation}

We say that a self-adjoint matrix $A(=A^\dagger) \in {\mathcal M}_\ell$ is {\em nonnegative} ({\em positive}) if the
corresponding quadratic form $x\mapsto x^\dagger A x$ is nonnegative definite (positive definite).
We denote the class of all nonnegative self-adjoint $\ell\times\ell$ matrices by
${\mathcal M}^+_\ell$. The corresponding
partial order is known as the {\em Loewner ordering\/} and is denoted by~$\succ$: $A\succ B$ means that $A-B$ is positive, i.e., $A-B\succ {\mathbf 0}$, and $A\succeq B$ means that $A-B\succeq {\mathbf 0}$, or
$A-B\in{\mathcal M}^+_\ell$. 

Here is the
first fact about the Loewner ordering that we will use later:

\begin{lemma}\label{triv87} Let ${\bf 0}\preceq A_j\preceq B_j$ for $j=1,\ldots k$. Then
$
{\bf 0}\preceq A_1+\cdots+A_k\preceq B_1+\cdots+B_k.
$
\end{lemma}
\begin{proof} Evaluate and compare the quadratic forms of both sums.
\end{proof}

We will also need the following result connecting
traces of self-adjoint matrices and their Loewner ordering:

\begin{lemma}\label{triv88}
Suppose $A\succeq B$ and $\tr A= \tr B$. Then $A=B$.
\end{lemma}
\begin{proof} By the linearity of traces, this is equalent
to the statement: Suppose $A\succeq {\bf 0}$ and $\tr A=0$. Then
$A={\bf 0}$. The latter follows from the fact that $\tr A = \sum_{j=1}^\ell
e_j^\dagger A e_j$, so if the trace of $A$ is zero, the action of $A$
on all standard unit vectors (hence on the entire space)
must be trivial.
\end{proof}

Another fact about traces we will need is the following:

\begin{lemma}\label{triv89}
 If $A\succ {\bf 0}$, then
$
\log \det(A)=\text{tr}(\log A).
$
\end{lemma}

\begin{proof} Without loss of generality, $A$ is a diagonal matrix with positive diagonal
elements, since the formula is invariant under unitary similarity. But
then $\log A$ is the diagonal matrix whose elements are the logarithms of
the diagonal elements of $A$. The conclusion of the Lemma is thus straightforward.
 \end{proof}

\begin{corollary}\label{triv90}
Let $A_1$, $\ldots$, $A_n\succ {\bf 0}$ and let $A:=A_1\cdots A_n \succ {\bf 0}$. Then
$$
\text{tr}\log(A_1\cdots A_n)=\log \det(A_1\cdots A_n)=\sum_{k=1}^n\text{tr}(\log A_k).
$$
\end{corollary}
\begin{proof} Apply Lemma~\ref{triv89} to the product $A=A_1\cdots A_n$.
\end{proof}

Finally, we will need the following interesting characterization
of the determinant via the trace (also used by Helson and Lowdenslager
in~\cite{HelsonLowdenslager}, also see, e.g.,~\cite[Exercise 19, p.486]{HornJohnson}):

\begin{lemma}\label{HL-Lemma} Let $\mathcal{A}$ be the set of all matrices in $\mathcal{M}_\ell$ with determinant $1$. Then every positive matrix $C$ satisfies
\begin{equation}\label{bestin54}
\inf_{A\in\mathcal{A}}\frac{1}{\ell}\,\text{tr}\left(ACA^\dagger\right)=[\det(C)]^{1/\ell}.
\end{equation}
\end{lemma}
\begin{proof} Let $U$ be a unitary matrix such that $C=UDU^\dagger$, where $D$ is the diagonal matrix with eigenvalues $\lambda_1,\ldots\lambda_\ell$. Then $\lambda=\det(U)\in\mathbb{T}$. It follows that $ACA^\dagger$=$(A\overline{\lambda}U)D(A\overline{\lambda}U)^\dagger$, implying that we may assume without loss of generality that $C=D$. Then
\begin{equation*}
\text{tr}(ADA^\dagger)=\lambda_1\Vert a_1\Vert^2+\lambda_2\Vert a_2\Vert^2+\cdots+\lambda_\ell\Vert a_\ell\Vert^2,
\end{equation*}
where $a_k$ denotes the $k$th column of $A$. By the arithmetic-geometric mean inequality,
\begin{equation}\label{AGM749}
\frac{\lambda_1\Vert a_1\Vert^2+\lambda_2\Vert a_2\Vert^2+\cdots+\lambda_\ell\Vert a_\ell\Vert^2}{\ell}\geq \sqrt[\ell]{\lambda_1\cdots\lambda_\ell\Vert a_1\Vert^2\cdots\Vert a_\ell\Vert^2}.
\end{equation}
By Hadamard's inequality \cite[Inequality 7.8.2]{HornJohnson},
\begin{equation}\label{Hadamard}
\Vert a_1\Vert\cdots\Vert a_\ell\Vert\geq \det(A)=1.
\end{equation}
The equality in \eqref{AGM749} occurs if and only if
\begin{equation*}
\lambda_1\Vert a_1\Vert^2=\cdots=\lambda_\ell\Vert a_\ell\Vert^2.
\end{equation*}
The equality in \eqref{Hadamard} occurs if and only if the columns $a_k$ form an orthogonal system in $\mathbb{C}^\ell$. It follows that the equality in \eqref{bestin54} is attained for the diagonal matrix $A$ with $a\lambda_1^{-1/2},\ldots,a\lambda_{\ell}^{-1/2}$ on the diagonal. Here $a$ is chosen so as to make the determinant of $A$ equal $1$.
\end{proof}

\section{Matrix measures}

A matrix-valued nonnegative measure $\mu$ on the unit circle ${\mathbb T}$ is a countably additive mapping of the Borel $\sigma$-algebra $\mathfrak{B}({\mathbb T})$ on ${\mathbb T}$ into the
set ${\mathcal M}_\ell^{+}$ of all nonnegative $\ell\times \ell$ matrices $\mu: B\mapsto\mu(B)\in {\mathcal M}_\ell^{+}$.
It follows that for any $E\in\mathfrak{B}({\mathbb T})$
\begin{equation*}
0\leq \mu(E)\leq \mu(E)+\mu(\mathbb{T}\setminus E)=\mu(\mathbb{T}).
\end{equation*}
Then $\nu(E)=\mu(\mathbb{T})^{-1/2}\mu(E)\mu(\mathbb{T})^{-1/2}$ is also a matrix-valued nonnegative measure which is called the normalization of $\mu$.
As before, we assume that $\mu$ is normalized: $ \mu({\mathbb T})={\bf 1}. $

Recall that  $\textsf{P}_{\ell}(\mathbb{T})$ denotes the set of all matrix probability
measures, i.e., the normalized matrix-valued nonnegative measures on $\mathbb{T}$.
Let $\{e_j\}_{j=1}^l$ be the standard basis in ${\mathbb C}^\ell$.
Then for every $E\in\mathfrak{B}({\mathbb T})$ we obtain the matrix of $\mu(E)$
\begin{equation}\label{matrixOfMeasure}
\mu(E)=\begin{pmatrix}\mu_{11}(E)&\mu_{12}(E)&\cdots&\mu_{1\ell}(E)\\
\mu_{21}(E)&\mu_{22}(E)&\cdots&\mu_{2\ell}(E)\\
\vdots&\vdots&\ddots&\vdots\\
\mu_{\ell 1}(E)&\mu_{\ell 2}(E)&\cdots&\mu_{\ell \ell}(E)
\end{pmatrix}.
\end{equation}
 Since $\pmb{|\!\!|}\alpha\pmb{|\!\!|}_1=\text{tr}(\alpha)$ for every $\alpha\in{\mathcal M}_\ell^{+}$, we see that
\begin{equation}\label{tracemeasure}
|\mu_{ij}(E)|=|(\mu(E)e_j,e_i)|\leq \Vert{\mu(E)}\Vert\leq\pmb{|\!\!|}\mu(E)\pmb{|\!\!|}_1=\text{tr}(\mu(E)).
\end{equation}
It follows that the entries $\mu_{ij}(E)$ of  $\mu(E)$ are finite complex measures on $\mathbb{T}$ which are absolutely continuous with respect to $\tr(\mu)$. Thus any element $\mu$
of $\textsf{P}_{\ell}(\mathbb{T})$ is nothing but a table of measures  \eqref{matrixOfMeasure} subject to positivity conditions  and domination by the $\tr(\mu)$. We say that $\mu\in\textsf{P}_{\ell}(\mathbb{T})$ is {\em absolutely continuous (discrete, singular)\/} if so is its trace measure. It follows that any $\mu\in\textsf{P}_{\ell}(\mathbb{T})$ can be uniquely decomposed into the sum
\begin{equation}\label{LebesgueDecomp}
\mu=\mu_a+\mu_d+\mu_s,
\end{equation}
where $\mu_a$ is absolutely continuous with respect to the Lebesgue measure $d\theta/(2\pi)$, $\mu_d$ is the discrete part of $\mu$ and $\mu_s$ is its singular part. Indeed, taking the Hahn-Lebesgue decomposition of the trace measure, we can associate three matrix-valued measures with it. Namely, the entries of $\mu_a$ are the absolutely continuous parts of $\mu_{ij}$ with respect to $\text{tr}(\mu)_a$, and similarly the enties of $\mu_d$ and $\mu_s$ for the discrete and singular parts, respectively. Since Borel supports of $\text{tr}(\mu)_a$, $\text{tr}(\mu)_d$, $\text{tr}(\mu)_s$ can be chosen to be disjoint, the positivity of the corresponding matrices follows immediately. Moreover $d\mu=\mathbf{M}(\mu, \zeta)\text{tr}(d\mu)$ where $\mathbf{M}(\mu,\zeta)\in{\mathcal{M}}_\ell^+$ for $\zeta\in\mathbb{T}$.

The measure $\mu_a$ can be found by Lebesgue's differentiation theorem:
\begin{equation}\label{difflebesgue}
\mu^{\prime}(e^{i\theta})=\lim_{\epsilon\rightarrow 0^+}\frac{\mu(I_\epsilon)}{2\epsilon}\quad\text{a.e. on } \mathbb{T},
\end{equation}
where $I_{\epsilon}$ denotes the arc of length $2\epsilon$ on $\mathbb{T}$ centered at $e^{i\theta}$. Then
\begin{equation*}
\mu_a(E)=\int_E\mu^{\prime}(e^{i\theta})\frac{d\theta}{2\pi}.
\end{equation*}

We say that a sequence $\{\mu^{(n)}\}_{n\geq 0}$ in $\textsf{P}_{\ell}(\mathbb{T})$ converges to $\mu\in\textsf{P}_{\ell}(\mathbb{T})$ in $*$-weak topology if
\begin{equation}\label{starconv47}
*{-}\lim_n \mu_{ij}^{(n)}=\mu_{ij}
\end{equation}
for every pair of indices $(i,j)$.  For our class
 $\textsf{P}_{\ell}(\mathbb{T})$, we need matrix analogues  
of two Helley's lemmas as they are stated in
\cite[Lemma 8.5, Theorem 8.6]{Khrushchev} for scalar measures:
\begin{theorem}\label{Helley1} If $*{-}\lim_n \mu^{(n)}=\mu$ in $\textsf{P}_{\ell}(\mathbb{T})$,
 then
\begin{equation}\label{arctest}
\lim_n \mu^{(n)}(I)=\mu(I)
\end{equation}
for any open arc $I$ on $\mathbb{T}$ such that $\mu$ vanishes at the endpoints of $I$.
\end{theorem}

\begin{proof} Let ${x}$ be an arbitrary fixed column-vector in $\mathbb{C}^\ell$; as usual,
 ${x}^\dagger$ denotes its conjugate transpose (row-vector). Let $t\mapsto f(t)$ be a nonnegative 
continuous function with values in $[0,1]$ supported on an open arc $I$. Then
\begin{equation}\label{mainineq538}
x^\dagger\mu^{(n)}(I)x=\int_I x^\dagger d\mu^{(n)}(t)x\geq \int_I f(t)x^\dagger d\mu^{(n)}(t)x.
\end{equation}
By \eqref{starconv47}
\begin{equation*}
\lim_n \int_I f(t)x^\dagger d\mu^{(n)}(t)x=\lim_n\int_I f(t)\sum_{i,j=1}^\ell\overline{x_i} \, d\mu_{ij}^{(n)} x_j=\sum_{i,j=1}^\ell\lim_n\int_I f(t)\overline{x_i} \,  d\mu_{ij}^{(n)}  x_j=\int_I f(t)x^\dagger d\mu(t)x.
\end{equation*}
This and \eqref{mainineq538} imply
\begin{equation}\label{firstineq245}
\limsup_n x^\dagger\mu^{(n)}(I)x\geq \liminf_n x^\dagger\mu^{(n)}(I)x\geq \sup_{f}\int_I f(t)x^\dagger d\mu(t)x=\int_Ix^\dagger d\mu(t)x=x^\dagger \mu(I) x.
\end{equation}
Similarly for the complementary arc $J$ to the closure of $I$ in $\mathbb{T}$ we have
\begin{equation}\label{secondineq245}
\limsup_n x^\dagger\mu^{(n)}(J)x\geq \liminf_n x^\dagger\mu^{(n)}(J)x\geq \sup_{f}\int_J f(t)x^\dagger d\mu(t)x=\int_Jx^\dagger d\mu(t)x=x^\dagger \mu(J) x.
\end{equation}
Since $\mu$ vanishes at the endpoints of $I$ and $\mu^{(n)}\in \textsf{P}_{\ell}(\mathbb{T})$ we obtain that
\begin{equation}\label{twoineq76}
\mu^{(n)}(I)+\mu^{(n)}(J)\preceq\mathbf{1}~, \qquad \mu(I)+\mu(J)=\mathbf{1}.
\end{equation}
Combining \eqref{firstineq245} and \eqref{secondineq245} with \eqref{twoineq76} we conclude that
\begin{multline*}
x^\dagger x=x^\dagger\mathbf{1}x\geq \limsup_n x^\dagger(\mu^{(n)}(I)+\mu^{(n)}(J))x
\geq \liminf_n x^\dagger(\mu^{(n)}(I)+\mu^{(n)}(J))x \\ \geq \int_Ix^\dagger d\mu(t)x+\int_Jx^\dagger d\mu(t)x=x^\dagger \mu(I) x+x^\dagger \mu(J) x=
x^\dagger \mu(\mathbb{T}) x=x^\dagger x.\\
\end{multline*}
This is only possible if equalities hold in \eqref{firstineq245} and in \eqref{secondineq245}.
Since the vector $x$ was arbitrary, this implies the conclusion of the theorem.
\end{proof}
\begin{theorem}\label{HelleysTh} Let $\{\mu^{(n)}\}_{n\geq 0}$ be a sequence  in $\textsf{P}_{\ell}(\mathbb{T})$ and $\mu\in\textsf{P}_{\ell}(\mathbb{T})$. Then $*{-}\lim_n \mu^{(n)}=\mu$ if and only if $\lim_n \mu^{(n)}(I)=\mu(I)$ for any open arc $I$ on $\mathbb{T}$ such that $\mu$ does not have point masses at the endpoints of $I$.
\end{theorem}
\begin{proof} One direction has been already proved in Theorem \ref{Helley1}.
Suppose now that $\lim_n \mu^{(n)}(I)=\mu(I)$ for any open arc $I$ on $\mathbb{T}$ such that $\mu$ does not have point masses at the endpoints of $I$. Then for every $x\in\mathbb{C}^\ell$, $\Vert x\Vert=1$, the sequence of usual probability measures $x^\dagger\mu^{(n)}x$ converges to $x^\dagger\mu x$ on any interval which does not have point masses of $\mu$ (and therefore of $x^\dagger\mu x$ since it is absolutely continuous with respect to $\mu$). Then by Theorem 8.6 of \cite{Khrushchev} the sequence $x^\dagger\mu^{(n)}x$ converges to $x^\dagger\mu x$ in the $*$-weak topology. Now the polarization identity implies $*{-}\lim_n x^\dagger\mu^{(n)}y=x^\dagger\mu y$ for any pair of vectors $(x,y)$. 
Setting $x:=e_i$, $y:=e_j$ for all pairs  $i,j$, we obtain the weak limits for all entries of $\mu$.
\end{proof}
\begin{theorem}\label{Ltwoth} Suppose that $\nu_{n}=h_{n}d\theta/(2\pi)$ where $h_{n}$ are matrix-valued functions on $\mathbb{T}$. Suppose that there is a positive constant $C$ such that
\begin{equation*}
\int_{\mathbb{T}}\Vert{h_{n}}\Vert^2\frac{d\theta}{2\pi}<C.
\end{equation*}
Then any $*$-weak limit point of $\{\nu_{n}\}$ is an absolutely continuous matrix-valued measure.
\end{theorem}
\begin{proof} By norm equivalence \eqref{ineqfornorms} in ${\mathcal M}_\ell$, we can replace the operator norm of $h_n$ by its Hilbert-Schmidt norm. For each matrix entry, the result of this
theorem is standard, so it holds for the Hilbert-Schmidt (and hence the spectral) norm of the entire matrix as well.
\end{proof}

Every $\mu\in\textsf{P}_{\ell}(\mathbb{T})$ defines two positive definite quadratic forms on the two-sided module $C(\mathbb{T}, {\mathcal M}_\ell)$ over ${\mathcal M}_\ell$  of all continuous functions with values in ${\mathcal M}_\ell$. They correspond to the right and left multiplication and are defined as matrix-valued `inner products' by
\begin{eqnarray}
\ang{f,g}_R:=\int f(x)^\dagger\, d\mu(x) g(x), \label{ileft} \\
\ang{f,g}_L:=\int g(x)\, d\mu(x) f (x)^\dagger. \label{iright}
\end{eqnarray}

Let ${\mathcal P}$ denote the set of all polynomials in $z\in{\mathbb C}$ with
coefficients from ${\mathcal M}_\ell$. For a nonnegative integer $n$,
${\mathcal P}_n$ will denote the set of polynomials in ${\mathcal P}$ of degree at most $n$.
Note that, to generate an infinite sequence of orthogonal
polynomials, $\mu$ must satisfy \eqref{matrixpoly} for every nonzero polynomial $p$.
This is equivalent to the condition that the non-negative Borel measure
\begin{equation*}
\det\left(\mathbf{M}(\mu,\zeta)\right)\text{tr}(d\mu)
\end{equation*}
has  infinite Borel support, see~\cite{DPS:08,Wiener}.

\section{Analysis of operator functions}

In this section we list some properties of the logarithm as an operator
function.
We start with the definitions of operator monotone, convex, and concave
functions
defined on the half real line $(0,\infty)$.
Let $\mathcal{H}$ be an infinite-dimensional (separable) Hilbert space.
Let $B_+(\mathcal{H})$ denote the set of all positive operators in
$B(\mathcal{H})$. A continuous real function $f$ on $(0,\infty)$ is said
to be {\em operator monotone\/} (or, more precisely, {\em operator monotone
increasing)\/} if
$A\preceq B$ implies $f(A)\preceq f(B)$ for $A,B\in B_+(\mathcal{H})$,
and {\em operator monotone decreasing\/} if $-f$ is operator monotone increasing,
i.e.,  if
$A\preceq B$ implies
$f(A)\succeq f(B)$, where $f(A)$
and $f(B)$ are defined via functional calculus as usual. Also, $f$ is
said to be
{\em operator convex\/} if $f(\lambda A+(1-\lambda)B)\preceq \lambda
f(A)+(1-\lambda)f(B)$ for
all $A,B\in B_+(\mathcal{H})$ and $\lambda\in(0,1)$, and {\em operator
concave} if $-f$ is
operator convex (see also ~\cite{Bhatia:96}).

One should not expect that the operator monotonicity and the operator
convexity of $f$ follow from
the same properties of the scalar function $f$. For example, a power
function $t^\alpha$ on $(0,\infty)$ is operator monotone if and only if
$\alpha\in[0,1]$,
operator monotone decreasing if and only if
$\alpha\in[-1,0]$, and operator convex if and only if
$\alpha\in[-1,0]\cup[1,2]$ (see, for instance,~\cite[Chapter V]{Bhatia:96}).
Moreover, the function $f(t)=\exp(t)$ is neither operator monotone nor
operator convex on any (spectral) interval.

As is known, the operator monotone functions are generated by
holomorphic functions that
map the upper half plane into the upper half plane. Clearly, if one
fixes a branch of the
logarithm so that it is real on $(0,\infty)$ then the corresponding
holomorophic function
maps the upper half plane into the upper half plane.

\begin{proposition}\label{OpMon}
The functions $\log t$ and $-1/t$ are operator monotone increasing on
$(0,\infty)$.
\end{proposition}
A detailed proof can be found in~\cite[Section~V.4]{Bhatia:96}.

So,  $1/t$ is operator monotone decreasing on
$(0,\infty)$. Furthermore,
it follows from~\cite[Exercise~V.3.14]{Bhatia:96} that the integration
of an operator monotone
decreasing function gives an operator concave function.

\begin{proposition}\label{OpConc}
The function $\log t$ is operator concave on $(0,\infty)$.
\end{proposition}

This statement can also be verified by means of~\cite[Theorem~3.1]{AnHi:11}.

Now, we are in a position to formulate the matrix Jensen inequality for
the logarithm.
Namely, Proposition~\ref{OpConc} and~\cite[Theorem~4.2]{FarZh:07} yield
the following statement.
\begin{proposition}\label{prJensen}
Let $f : {\mathbb T} \to B_+({\mathcal M}_\ell)$ be a measurable function.
Then the following inequality holds:
\begin{equation}\label{MJIN}
\int_{\mathbb{T}}\log f(\theta)
\frac{d\theta}{2\pi}
\preceq \log\left(\int_{\mathbb{T}}f(\theta)\frac{d\theta}{2\pi}\right).
\end{equation}
\end{proposition}

Besides monotonicity and convexity, we will also deal with  operator continuity. Recall that
a function $f$ defined on $(0,\infty)$ is operator continuous if the
relation
$\Vert A_n-A\Vert_{\mathcal{H}}\to 0$ implies $\Vert
f(A_n)-f(A)\Vert_{{\mathcal H}}\to 0$
for any $A, A_n\in B_+({\mathcal H})$.

\begin{proposition}\label{OpCont}
The function $\log t$ is operator continuous on $(0,\infty)$.
\end{proposition}
\begin{proof}
Since $\log t$ can be extended to a holomorphic function on $\mathbb{C}\setminus (-\infty, 0)$, the statement follows directly from the Dunford-Schwarz operator calculus.
\end{proof}

\section{Matrix orthogonal polynomials on the unit circle}

We begin by recalling some basic facts from~\cite{DPS:08} for
 convenience of the reader.
Let $\sigma\in\textsf{P}_{\ell}(\mathbb{T})$ be a matrix probability measure such that $\det(\mathbf{M}(\sigma,\zeta))\text{tr}(d\sigma(\zeta))$ has infinite Borel support. We define right and left monic orthogonal matrix polynomials $\Phi_n^R, \Phi_n^L$ by
applying the Gram--Schmidt procedure in $C(\mathbb{T}, {\mathcal M}_\ell)$ with respect to the `inner products' \eqref{ileft} and  \eqref{iright} to the sequence
$\{{\bf 1}, z {\bf 1},  z^2 {\bf 1}, \ldots \}$.
In other words,  $\Phi_n^R$ is the unique
matrix polynomial $z^n {\bf 1} + \textit{lower order terms}$ \/ satisfying
the orthogonality conditions
\begin{equation} \label{ortR}
0=\ang{z^k{\bf 1}, \Phi_n^R}_R =\int \left(z^k{\bf 1}\right)^\dagger\, d\sigma(x) \Phi_n^R, \qquad k=0,1,\dots, n-1.
\end{equation}
Similarly $\Phi_n^L$ is the unique
matrix polynomial $z^n {\bf 1} + \textit{lower order terms}$ \/ satisfying
\begin{equation}\label{leftmatrixpoly}
0=\ang{z^k{\bf 1},\Phi_n^L}_L:=\int \Phi_n^L\, d\sigma(x) \left(z^k{\bf 1}\right)^\dagger, \qquad k=0,1,\dots, n-1.
\end{equation}

The normalized orthogonal matrix polynomials are defined by
\begin{equation}\label{monicandnorm}
\varphi_0^L=\varphi_0^R={\bf 1},\qquad
\varphi_n^L = \kappa_n^L \Phi_n^L \qquad \text{ and } \quad \varphi_n^R =
\Phi_n^R\kappa_n^R
\end{equation}
where the $\kappa$'s are defined according to the normalization
conditions
\begin{equation}\label{ortcond}
\ang{\varphi_n^R,\varphi_m^R}_R=\delta_{nm}{\bf 1},
\qquad
\ang{\varphi_n^L,\varphi_m^L}_L=\delta_{nm}{\bf 1},
\end{equation}
along with the following positivity conditions:
\begin{equation}\label{NormCond}
\kappa_{n+1}^L (\kappa_n^L)^{-1} \succ {\bf 0} \quad \text{and} \quad
(\kappa_n^R)^{-1}
\kappa_{n+1}^R \succ {\bf 0}.
\end{equation}
Note that the $\kappa_{n}^L$ are determined by the normalization
condition up to multiplication on the left by unitary matrices.
It can be shown that these
unitaries can always be uniquely chosen so as to satisfy
\eqref{NormCond}, see~\cite{DPS:08}.

Now define
\[
\rho_n^L :=\kappa_n^L (\kappa_{n+1}^L)^{-1} \qquad \text{and} \quad
\rho_n^R :=
(\kappa_{n+1}^R)^{-1} \kappa_n^R.
\]
Being inverses of positives matrices, $\rho_n^L$ and $
\rho_n^R $ are positive definite as well.
In particular, we have that
\begin{equation}\label{kappaformula}
\kappa_n^L = (\rho_{0}^L \cdots \rho_{n-1}^L)^{-1} \quad \text{and}
\quad \kappa_n^R =
(\rho_{n-1}^R \cdots \rho_{0}^R)^{-1}.
\end{equation}

In the matrix case as well as in the scalar case we have the Szeg\H{o} recursion.
Before stating it,  we recall that, for a matrix polynomial $P_n$ of
degree $n$, we define the reversed polynomial $P_n^*$ by~(\ref{reverseP}):
$P_n^*(z) = z^n P_n(1/\bar{z})^\dagger.$

\begin{theorem}[\cite{DPS:08}]\label{szegoth}
There is a sequence of contractive matrices $\alpha_n$ in ${\mathcal M}_\ell$
such that
\begin{eqnarray}
z \varphi_n^L - \rho_n^L \varphi_{n+1}^L & = \alpha_n^\dagger
\varphi_n^{R,*},  \label{S1}\\
z \varphi_n^R - \varphi_{n+1}^R \rho_n^R & = \varphi_n^{L,*}
\alpha_n^\dagger, \label{S2}
\end{eqnarray}
where $\rho_n^L$ and $\rho_n^R$ are defined as follows
\begin{equation}\label{rhoalpha4}
\rho_n^L = ({\bf 1} - \alpha_n^\dagger \alpha_n)^
{1/2},\qquad \rho_n^R = ({\bf 1} -\alpha_n \alpha_n^\dagger)^{1/2}.
\end{equation}
\end{theorem}
Setting $z=0$ in \eqref{S1} and using \eqref{monicandnorm}, we derive the following formulas for the parameters:
\begin{equation}\label{parametersformula}
\alpha_n=-(\kappa_n^R)^{-1}\mathbf{\Phi}^L_{n+1}(0)^\dagger (\kappa_n^L)^\dagger=
-(\kappa_n^R)^\dagger\mathbf{\Phi}^R_{n+1}(0)^\dagger(\kappa_n^L)^{-1} .
\end{equation}
Alternatively, one can also set $z=0$ in formulas (3.11) of \cite{DPS:08}.
\begin{lemma}\label{LeftAndRight} The left and right monic orthogonal polynomials of $\overline{\sigma}$ and $\sigma$ are related by
\begin{equation}\label{monicrel}
\mathbf{\Phi}_n^L(e^{i\theta},\overline{\sigma}){=}\mathbf{\Phi}_n^R(e^{-i\theta},\sigma)^\dagger.
\end{equation}
\end{lemma}
\begin{proof} For $k<n$ we have,  by \eqref{ortR},
\begin{multline*}
0 =  \left(\int \left(z^k{\bf 1}\right)^\dagger\, d\sigma(z) \Phi_n^R(z,\sigma)\right)^\dagger =
\int \left(\Phi_n^R(z,\sigma)\right)^\dagger\, d\sigma(z) \left(z^k{\bf 1}\right)=
\int \left(\Phi_n^R(\overline{z},\sigma)\right)^\dagger\, d\sigma(\overline{z}) \left(\overline{z}^k{\bf 1}\right)
\\=\int \left(\Phi_n^R(\overline{z},\sigma)\right)^\dagger\, d\overline{\sigma}({z}) \left({z}^k{\bf 1}\right)^\dagger=\int \Phi_n^L({z},\overline{\sigma})\, d\overline{\sigma}({z}) \left({z}^k{\bf 1}\right)^\dagger,
\end{multline*}
which implies \eqref{monicrel} by \eqref{leftmatrixpoly}.
\end{proof}

\begin{proposition}\label{parametersofsymmeasure} If $\{\alpha_k\}_{k\geq 0}$ are the parameters of $\sigma$, then $\{\alpha_k^\dagger\}_{k\geq 0}$ are the parameters of $\overline{\sigma}$.
\end{proposition}
\begin{proof} By \eqref{monicrel}, the matrix coefficients of the polynomial
$\mathbf{\Phi}_n^L(e^{i\theta},\overline{\sigma})$ are the matrices adjoint  to the coefficients of the polynomial $\mathbf{\Phi}_n^R(e^{i\theta},\sigma)$. In particular,
\begin{equation}\label{4589}
\mathbf{\Phi}_{n+1}^L(0,\overline{\sigma})=\mathbf{\Phi}_{n+1}^R(0,\sigma)^\dagger.
\end{equation}

Since $\kappa_0^R=\kappa_0^L=\mathbf{1}$, we see that
\begin{equation*}
\alpha_0(\overline{\sigma})=-\mathbf{\Phi}^L_{1}(\overline{0,\sigma})^\dagger=
-\mathbf{\Phi}^R_{1}({0,\sigma})=\alpha_0(\sigma)^\dagger.
\end{equation*}
Suppose that we already proved that $\alpha_k(\overline{\sigma})=\alpha_k(\sigma)^\dagger$ for $k<n$. Then, by the induction hypothesis and by \eqref{kappaformula},
\begin{equation}\label{leftformright5063}
\kappa_n^R(\overline{\sigma})=\kappa_n^L({\sigma})^\dagger,\qquad \kappa_n^L(\overline{\sigma})=\kappa_n^R({\sigma})^\dagger.
\end{equation}
It follows that
\begin{multline*}
\alpha_n(\overline{\sigma})=-(\kappa_n^R(\overline{\sigma}))^{-1}\mathbf{\Phi}_{n+1}^L
(0,\overline{\sigma})^\dagger(\kappa_n^L(\overline{\sigma}))^\dagger=
-(\kappa_n^L({\sigma})^\dagger)^{-1}\mathbf{\Phi}_{n+1}^R(0,\sigma)\kappa_n^R({\sigma})\\=
\left(-\kappa_n^R({\sigma})^\dagger\mathbf{\Phi}_{n+1}^R(0,\sigma)^\dagger(\kappa_n^L({\sigma}))^{-1}\right)^\dagger=
\alpha_n(\sigma)^\dagger,
\end{multline*}
see \eqref{parametersformula}.
\end{proof}
\begin{corollary}\label{eqnumberfive8} The left and right orthogonal polynomials are related by 
formula $\eqref{lefttoright45}$.
\end{corollary}
\begin{proof} We have
\begin{equation*}
\varphi_n^L(e^{i\theta}, \overline{\sigma})=\kappa_n^L(\overline{\sigma})\mathbf{\Phi}_n^L(e^{i\theta}, \overline{\sigma})=\kappa_n^R(\sigma)^\dagger\mathbf{\Phi}_n^R(e^{-i\theta},{\sigma})^\dagger=
\varphi_n^R(e^{-i\theta},{\sigma})^\dagger.
\end{equation*}
\end{proof}


We next recall the notion of Bernstein-Szeg\H{o} approximation.
We begin with a list of properties of matrix orthogonal polynomials.
\begin{theorem}[{\cite[Theorem 3.8]{DPS:08}}] \label{PropMOP}
The polynomials $\varphi^L$, $\varphi^R$ satisfy the following conditions:
\begin{enumerate}
\item[(i)] For $z\in{\mathbb T}$, all of $\varphi_n^{R,*}(z)$,
$\varphi_n^{L,*}(z)$, $\varphi_n^R(z)$,
$\varphi_n^L(z)$ are invertible.

\item[(ii)] For $z\in{\mathbb D}$, $\varphi_n^{R,*}(z)$ and
$\varphi_n^{L,*}(z)$ are invertible.

\item[(iii)] For any $z\in{\mathbb T}$,
\begin{equation}\label{LeftRight}
\varphi_n^R(z) \varphi_n^R(z)^\dagger =
\varphi_n^L (z)^\dagger
\varphi_n^L (z)~.
\end{equation}
\end{enumerate}
\end{theorem}

Given a finite sequence $\{\alpha_j\}_{j=0}^{n-1}$ of contractive matrices,
we can always use the Szeg\H{o} recursion
to define the polynomials $\varphi_j^R, \varphi_j^L$ for $j=0,1,\dots, n$.
Analogously to the scalar case, let us define a measure
$d\mu_n$ on ${\mathbb T}$ by
\begin{equation}\label{eqszego1}
d\mu_n(\theta) = [\varphi_n^R(e^{i\theta})\varphi_n^R (e^{i\theta})^
\dagger]^{-1}
\, \frac{d\theta}{2\pi}.
\end{equation}
In view of~\eqref{LeftRight}, we also see that
\begin{equation}\label{eqszego2}
d\mu_n(\theta)= [\varphi_n^L (e^{i\theta})^\dagger \varphi_n^L (e^{i
\theta})]^{-1} \, \frac{d\theta}{2\pi}\, .
\end{equation}
Also, directly from the definition of the right orthogonal polynomials, we have
\begin{equation}\label{eqszego3}
d\mu_n(\theta)= [\varphi_n^{R,*} (e^{i\theta})^\dagger \varphi_n^{R,*} (e^{i
\theta})]^{-1} \, \frac{d\theta}{2\pi}.
\end{equation}
The measure $d\mu_n$ in~(\ref{eqszego3}) is called the {\em right Bernstein-Szeg\H{o} approximation\/}
to $\sigma$. The {\em left   Bernstein-Szeg\H{o} approximation\/} to $\sigma$ is given by
\begin{equation}\label{eqszego32}
d\mu_n^L(\theta)= [\varphi_n^{L,*} (e^{i\theta})\varphi_n^{L,*} (e^{i
\theta})^\dagger ]^{-1} \, \frac{d\theta}{2\pi}.
\end{equation}

Now we are in a position to formulate the main result of this section.

\begin{theorem}[\cite{DPS:08}]\label{BSconv}
The matrix-valued measure $d\mu_n$ is normalized
and its right matrix orthogonal polynomials for $j=0,\dots, n$ are
$\{\varphi_j^R\}_{j=0}^n$.
The Verblunsky coefficients for $d\mu_n$ are
\begin{equation}
\alpha_j (d\mu_n) = \begin{cases} \alpha_j, & j\leq n, \\
\bf{0}, & j\geq n+1.
\end{cases}
\end{equation}
Moreover, $*{-}\lim_{n \to \infty} d\mu_n =d \sigma$.
\end{theorem}

Following \cite{DPS:08},  we associate the matrix
\begin{equation}\label{alphamatrix}
A^L(\alpha,z)=\begin{pmatrix}z(\rho^L)^{-1}&-(\rho^L)^{-1}\alpha^\dagger\\
-z(\rho^R)^{-1}\alpha&(\rho^R)^{-1}\end{pmatrix}
\end{equation}
to a given matrix parameter $\alpha$. Then
\begin{equation}\label{it6}
\begin{pmatrix}
\varphi_{n}^L\\
\varphi_{n}^{R,*}
\end{pmatrix}=A^L(\alpha_{n-1},z)\cdots A^L(\alpha_{0},z)
\begin{pmatrix}
\mathbf{1}\\
\mathbf{1}
\end{pmatrix} .
\end{equation}
Applying the adjoint $\dagger$ to both sides and taking the product
over $\alpha_j$ for $j=0, \ldots, n-1$, we obtain
\begin{equation*}
{\varphi_{n}^L}^{\dagger}\varphi_{n}^L+
{\varphi_{n}^{R,*}}^{\dagger}
\varphi_{n}^{R,*}=
\begin{pmatrix}
\mathbf{1}&
\mathbf{1}
\end{pmatrix}
{A^L(\alpha_{0},z)}^{\dagger}\cdots
{A^L(\alpha_{n-1},z)}^{\dagger}
A^L(\alpha_{n-1},z)\cdots A^L(\alpha_{0},z)
\begin{pmatrix}
\mathbf{1}\\
\mathbf{1}
\end{pmatrix} .
\end{equation*}
Note that \eqref{eqszego2} and \eqref{eqszego3} imply that the equality  \begin{equation}\label{importanteq}{\varphi_{n}^L}^{\dagger}\varphi_{n}^L=
{\varphi_{n}^{R,*}}^{\dagger}
\varphi_{n}^{R,*}
\end{equation}
holds  on the circle $\mathbb{T}$, implying that
\begin{equation}\label{idforBernstein}
{\varphi_{n}^{R,*}}^{\dagger}
\varphi_{n}^{R,*}=\frac{1}{2}\begin{pmatrix}
\mathbf{1}&
\mathbf{1}
\end{pmatrix}
{A^L(\alpha_{0},z)}^{\dagger}\cdots
{A^L(\alpha_{n-1},z)}^{\dagger}
A^L(\alpha_{n-1},z)\cdots A^L(\alpha_{0},z)
\begin{pmatrix}
\mathbf{1}\\
\mathbf{1}
\end{pmatrix} .
\end{equation}

Using the fact $\rho^R \alpha = \alpha \rho^L$, 
the matrix in \eqref{alphamatrix} can be factored as follows:
\begin{equation}\label{factoralpha}
A^L(\alpha,z)=\begin{pmatrix}(\rho^L)^{-1}&0\\0&(\rho^R)^{-1}\end{pmatrix}
\begin{pmatrix}z\mathbf{1}&-\alpha^\dagger\\-z\alpha&\mathbf{1}\end{pmatrix}=
\begin{pmatrix}z\mathbf{1}&-\alpha^\dagger\\-z\alpha&\mathbf{1}\end{pmatrix}\begin{pmatrix}(\rho^L)^{-1}&0\\0&(\rho^R)^{-1}\end{pmatrix}.
\end{equation}
To factor the non-diagonal matrix in \eqref{factoralpha}, we apply Schur's factorization
\begin{equation}\label{schurfact}
\begin{pmatrix}A&B\\C&D\end{pmatrix}=\begin{pmatrix}\mathbf{1}&0\\CA^{-1}&\mathbf{1}\end{pmatrix}
\begin{pmatrix}A&0\\0&D-CA^{-1}B\end{pmatrix}
\begin{pmatrix}\mathbf{1}&A^{-1}B\\0&\mathbf{1}\end{pmatrix}
\end{equation}
with
\begin{equation*}
\begin{matrix}A=z\mathbf{1}&B=-\alpha^\dagger\\C=-z\alpha&D=\mathbf{1}\end{matrix} .
\end{equation*}
Then
\begin{equation}\label{4factor}
A^L(\alpha,z)=\begin{pmatrix}(\rho^L)^{-1}&0\\0&(\rho^R)^{-1}\end{pmatrix}
\begin{pmatrix}\mathbf{1}&0\\-\alpha&\mathbf{1}\end{pmatrix}
\begin{pmatrix}z\mathbf{1}&0\\0&\mathbf{1}-\alpha\alpha^\dagger\end{pmatrix}
\begin{pmatrix}\mathbf{1}&-(z)^{-1}\alpha^\dagger\\0&\mathbf{1}\end{pmatrix}.
\end{equation}
Since $\det \rho^L=\det \rho^R$, we conclude that
\begin{equation}\label{detofA}
\det A^L(\alpha, z)=z^{\ell}.
\end{equation}
It is easy to check that
\[
\varphi_1^L(z)=(\rho_0^L)^{-1}(z-\alpha_0^\dagger),
\qquad
\varphi_1^R(z)=(z-\alpha_0^\dagger)(\rho_0^R)^{-1}.
\]
Further, forming the Szeg\H{o} dual, we obtain
\[
\varphi_1^{L,*}(z)=({\bf 1}-z\alpha_0)(\rho_0^L)^{-1},
\qquad
\varphi_1^{R,*}(z)=(\rho_0^R)^{-1}({\bf 1}-z\alpha_0).
\]
After pertinent multiplications, this produces
\[
\varphi_1^{R,*} (z)^\dagger \varphi_1^{R,*}(z)=
({\bf 1}-\overline{z}\alpha_0^\dagger)(\rho_0^R)^{-2}({\bf 1}-z\alpha_0),
\]

\[
\varphi_1^{L,*} (z) \varphi_1^{L,*}(z)^\dagger=
({\bf 1}-{z}\alpha_0)(\rho_0^L)^{-2}({\bf 1}-\overline{z}\alpha_0^\dagger).
\]
\begin{proposition}\label{formulaneq1} For every $\alpha\in{\mathcal M}_\ell$ and $z\in\mathbb{T}$
\begin{equation*}
({\bf 1}-\overline{z}\alpha^{\dag})({\bf 1}-\alpha\alpha^{\dag})^{-1}({\bf 1}-z\alpha)=\left[({\bf 1}-\overline{z}\alpha^{\dag})^{-1}+({\bf 1}-z\alpha)^{-1}-{\bf 1}\right]^{-1}.
\end{equation*}
\end{proposition}
\begin{proof}
Consider the matrix polynomial
\begin{equation}\label{poly}
p(z)=({\bf 1}-\overline{z}\alpha^{\dag})({\bf 1}-\alpha\alpha^{\dag})^{-1}({\bf 1}-z\alpha).
\end{equation}
Since $z\overline{z}=1$, we have
\begin{equation}\label{first}
{\bf 1}-\alpha\alpha^{\dag}={\bf 1}-z\overline{z}\alpha\alpha^{\dag}=({\bf 1}-z\alpha)({\bf 1}+\overline{z}\alpha^{\dag})+z\alpha-\overline{z}\alpha^{\dag}.
\end{equation}
Similarly,
\begin{equation}\label{second}
{\bf 1}-\alpha\alpha^{\dag}={\bf 1}-z\overline{z}\alpha\alpha^{\dag}=({\bf 1}+z\alpha)({\bf 1}-\overline{z}\alpha^{\dag})-z\alpha+\overline{z}\alpha^{\dag}.
\end{equation}
The sum of the expressions~\eqref{first} and~\eqref{second} yields
\begin{equation}\label{main}
\begin{array}{ccl}
2({\bf 1}-\alpha\alpha^{\dag}) &= &({\bf 1}+z\alpha)({\bf 1}-\overline{z}\alpha^{\dag})+({\bf 1}-z\alpha)({\bf 1}+\overline{z}\alpha^{\dag}) \\

&=&({\bf 2}-({\bf 1}-z\alpha))({\bf 1}-\overline{z}\alpha^{\dag})+({\bf 1}-z\alpha)({\bf 2}-({\bf 1}-\overline{z}\alpha^{\dag})).
\end{array}
\end{equation}
Let us denote ${\bf 1}-z\alpha$ by $B$ for brevity. From~\eqref{poly} and~\eqref{main} we obtain
\begin{equation*}\label{poly.2}
\begin{array}{ccl}
p(z) &= &B^{\dag}({\bf 1}-\alpha\alpha^{\dag})^{-1}B \; =\; 2B^{\dag}\left[B({\bf 2}-B^{\dag})+({\bf 2}-B)B^{\dag}\right]^{-1}B  \\

&=&2\left[({\bf 2}-B^{\dag})B^{-\dag}+B^{-1}({\bf 2}-B)\right]^{-1} \; = \; 2\left[2B^{-\dag}-{\bf 2}+2B^{-1}\right]^{-1}\\

&=&\left[({\bf 1}-\overline{z}\alpha^{\dag})^{-1}+({\bf 1}-z\alpha)^{-1}-{\bf 1}\right]^{-1}.
\end{array}
\end{equation*}

\end{proof}

\section{The Bernstein--Szeg\H{o} approximation}

In this section we obtain a formula for the Bernstein--Szeg\H{o} approximation of a matrix
 probability measure.
\begin{lemma}\label{neglemma7} Let $\beta_n$ be the matrix defined by
\begin{equation}\label{an89}
\beta_n=\exp\int_{0}^{2\pi}\log([\varphi_n^{R,*} (e^{i\theta})^\dagger \varphi_n^{R,*} (e^{i
\theta})]^{-1})\frac{d\theta}{2\pi}.
\end{equation}
Then $\log\beta_n$ is self-adjoint and nonpositive.
\end{lemma}
\begin{proof} Since $[\varphi_n^{R,*} (e^{i\theta})^\dagger \varphi_n^{R,*} (e^{i
\theta})]^{-1}$ is a positive matrix for every $\theta$ its logarithm is self-adjoint as well as the integral $\log\beta_n$ of the logarithm of this matrix. By Proposition \ref{prJensen}, \begin{equation*}
\log\beta_n=\int_{0}^{2\pi}\log([\varphi_n^{R,*} (e^{i\theta})^\dagger \varphi_n^{R,*} (e^{i
\theta})]^{-1})\frac{d\theta}{2\pi}\preceq\log\left(\int_{0}^{2\pi}[\varphi_n^{R,*} (e^{i\theta})^\dagger \varphi_n^{R,*} (e^{i
\theta})]^{-1})\frac{d\theta}{2\pi}\right)=\log\mu_n(\mathbb{T})=\mathbf{0}
\end{equation*}
since $\mu_n$ is normalized: $\mu_n(\mathbb T)={\bf 1}$.
\end{proof}

The standard operator calculus and Lemma \ref{neglemma7} imply that  the matrix $\beta_n=\exp \log\beta_n$ is self-adjoint and satisfies
\begin{equation}\label{bettaop9}
\mathbf{0}\prec \beta_n\preceq \mathbf{1}.
\end{equation}
\begin{lemma}\label{detid0} For $\beta_n$, we have
\begin{equation}\label{ident65}
\log\det{\beta_n}=\text{tr}\log\beta_n=\log\prod_{k=0}^{n-1}\det\left({\bf 1} -\alpha_k \alpha_k^\dagger\right)=\sum_{k=0}^{n-1}\text{tr}\log\left({\bf 1} -\alpha_k \alpha_k^\dagger\right).
\end{equation}
\end{lemma}
\begin{proof} Applying elementary transformations and formula~\eqref{rhoalpha4}, we obtain
\begin{multline*}
\text{tr}(\log\beta_n)=\int_{0}^{2\pi}\text{tr}\log([\varphi_n^{R,*} (e^{i\theta})^\dagger \varphi_n^{R,*} (e^{i
\theta})]^{-1})\frac{d\theta}{2\pi}=\int_{0}^{2\pi}\log\det([\varphi_n^{R,*} (e^{i\theta})^\dagger \varphi_n^{R,*} (e^{i
\theta})]^{-1})\frac{d\theta}{2\pi}\\=\int_{0}^{2\pi}\log\det([\varphi_n^{R,*} (e^{i\theta})^\dagger]^{-1})\frac{d\theta}{2\pi}+\int_{0}^{2\pi}\log\det([\varphi_n^{R,*} (e^{i\theta})]^{-1})\frac{d\theta}{2\pi}=2\textbf{Re}\int_{0}^{2\pi}\log\det([\varphi_n^{R,*} (e^{i\theta})]^{-1})\frac{d\theta}{2\pi}\\=2\textbf{Re}\log\det([\varphi_n^{R,*} (0)]^{-1})=
2\log\det\left(\rho_{0}^R\cdots\rho_{n-1}^R\right)=\log\prod_{k=0}^{n-1}\det\left({\bf 1} -\alpha_k \alpha_k^\dagger\right)
\end{multline*}
since the function $z\mapsto\log\det\left([\varphi_n^{R,*} (z)]^{-1}\right)$ is analytic in the closed unit disc.
\end{proof}

\vskip 1mm \noindent
{\bf Remark.\/} In general, $\log (AB)$ cannot be written as $\log A + \log B$ if  
$A$ and $B$ are matrices. So, the integral in Lemma~\ref{neglemma7} cannot be evaluated 
by the mean value theorem. In other words, the function
 $\log([\varphi_n^{R,*} (e^{i\theta})^\dagger \varphi_n^{R,*} (e^{i
\theta})]^{-1})$ is in general not a restriction of a harmonic function
to the unit circle. Our next lemma addresses the easy case when the logarithm
in question does split. \vskip 2mm

If $\{\alpha_1,\ldots,\alpha_{n-1}\}$ is a commuting family of normal matrices,
then the self-adjoint matrix $\beta_n$ can be evaluated explicitly as follows:
\begin{lemma}\label{directcomp} Let $\{\alpha_1,\ldots,\alpha_{n-1}\}$ be a commuting family of normal matrices. Then
\begin{equation*}
\beta_n=\prod_{k=0}^{n-1}\left({\bf 1} -\alpha_k \alpha_k^\dagger\right).
\end{equation*}
\end{lemma}
\begin{proof} The proof follows the proof of Lemma \ref{detid0}
 since in this case $\varphi_n^{R,*} (z)$ is a normal matrix for any value of $z$.
\end{proof}

\section{The Matrix Szeg\"o Theorem}

\begin{definition}
A matrix probability measure $\sigma\in\textsf{P}_{\ell}(\mathbb{T})$ is said to be a Szeg\"o measure if
\begin{equation}\label{defSz}
\int_{\mathbb{T}}\tr\log \sigma^{\prime}\,\frac{d\theta}{2\pi}>-\infty.
\end{equation}
\end{definition}
\begin{theorem}\label{firstineq}For any matrix probability measure $\sigma\in\textsf{P}_{\ell}(\mathbb{T})$ and any $n\in {\mathbb N}$,
\begin{equation}\label{HalfEq}
\int_{\mathbb{T}}\tr\log \sigma^{\prime}\,\frac{d\theta}{2\pi}\leq
\tr\log\beta_n=\log\prod_{k=0}^{n-1}\det(1-\alpha_k^\dagger \alpha_k).
\end{equation}
\end{theorem}
\begin{proof}If
$\int_{\mathbb{T}}\tr\log \sigma^{\prime}\,\frac{d\theta}{2\pi}=-\infty$,
the conclusion of the theorem holds trivially, so assume that $\sigma$
is a Szeg\H{o} measure, i.e., the corresponding integral is not $-\infty$.
Jensen's matrix inequality
from Proposition \ref{prJensen} implies
\begin{equation}\label{GPOP519}
\int_{\mathbb{T}}\log
\left(\beta_n^{1/2}[\varphi_n^{R,*} (e^{i\theta})^\dagger \sigma^{\prime}\varphi_n^{R,*} (e^{i
\theta})\beta_n^{1/2}]\right)
\frac{d\theta}{2\pi}
\preceq \log\left(\int_{\mathbb{T}}\beta_n^{1/2}
[\varphi_n^{R,*} (e^{i\theta})^\dagger \sigma^{\prime}\frac{d\theta}{2\pi}\varphi_n^{R,*} (e^{i
\theta})]\beta_n^{1/2}\right).
\end{equation}
Now we replace $\sigma^{\prime}$ by $\sigma$, which is larger in the Loewner ordering, according to \eqref{LebesgueDecomp}. Since $\sigma$ is absolutely continuous with respect of $\tr(\sigma)$, there exist two disjoint Borel sets $E$ and $F$ and a Borel matrix function $x\mapsto \texttt{M}(x)$ such that
\begin{equation*}
d\sigma =\texttt{M}\chi_E\tr(\sigma_a)+\texttt{M}\chi_F\tr(d\sigma_d+d\sigma_s)\,
\end{equation*}
where $E$ is a Borel support of $\tr(d\sigma_a)$, $F$ is a Borel support of $\tr(d\sigma_d+d\sigma_s)$ and $\chi_E$, $\chi_F$ are the indicators of $E$ and $F$ correspondingly. Notice that
\begin{equation}\label{eqer57}
\begin{aligned}
\beta_n^{1/2}
[\varphi_n^{R,*} (e^{i\theta})^\dagger \texttt{M}\chi_E\tr(\sigma_a)\varphi_n^{R,*} (e^{i
\theta})]\beta_n^{1/2}&=\beta_n^{1/2}
[\varphi_n^{R,*} (e^{i\theta})^\dagger \sigma^{\prime}\frac{d\theta}{2\pi}\varphi_n^{R,*} (e^{i
\theta})]\beta_n^{1/2};\\
\beta_n^{1/2}[\varphi_n^{R,*} (e^{i\theta})^\dagger \texttt{M}\chi_F\tr(d\sigma_d+d\sigma_s)\varphi_n^{R,*} (e^{i
\theta})]\beta_n^{1/2}&=\beta_n^{1/2}
[\varphi_n^{R,*} (e^{i\theta})^\dagger (d\sigma_d+d\sigma_s)\varphi_n^{R,*} (e^{i
\theta})]\beta_n^{1/2}.
\end{aligned}
\end{equation}
Combining \eqref{eqer57} with the result of Lemma \ref{triv87}, we obtain
\begin{multline*}
\int_{\mathbb{T}}\beta_n^{1/2}
[\varphi_n^{R,*} (e^{i\theta})^\dagger \sigma^{\prime}\frac{d\theta}{2\pi}\varphi_n^{R,*} (e^{i
\theta})]\beta_n^{1/2}=\int_{\mathbb{T}}\beta_n^{1/2}
[\varphi_n^{R,*} (e^{i\theta})^\dagger \texttt{M}\chi_E\tr(\sigma_a)\varphi_n^{R,*} (e^{i
\theta})]\beta_n^{1/2}  \\
\preceq \int_{\mathbb{T}}\beta_n^{1/2}
[\varphi_n^{R,*} (e^{i\theta})^\dagger \texttt{M}\chi_E\tr(\sigma_a)\varphi_n^{R,*} (e^{i
\theta})]\beta_n^{1/2}+\int_{\mathbb{T}}[\varphi_n^{R,*} (e^{i\theta})^\dagger \texttt{M}\chi_F\tr(d\sigma_d+d\sigma_s)\varphi_n^{R,*} (e^{i
\theta})]\beta_n^{1/2}\\=
\int_{\mathbb{T}}\beta_n^{1/2}
[\varphi_n^{R,*} (e^{i\theta})^\dagger d\sigma\varphi_n^{R,*} (e^{i
\theta})]\beta_n^{1/2}.
\end{multline*}
By \eqref{GPOP519} and by the operator monotonicity of the logarithm from Lemma \ref{OpMon}, we obtain
\begin{multline}\label{importantineq}
\int_{\mathbb{T}}\log
\left(\beta_n^{1/2}[\varphi_n^{R,*} (e^{i\theta})^\dagger \sigma^{\prime}\varphi_n^{R,*} (e^{i
\theta})]\beta_n^{1/2}\right)
\frac{d\theta}{2\pi}
\preceq\log\left(\int_{\mathbb{T}}\beta_n^{1/2}
[\varphi_n^{R,*} (e^{i\theta})^\dagger{d\sigma} \varphi_n^{R,*} (e^{i
\theta})]\beta_n^{1/2}\right)\\=\log\left(\beta_n^{1/2}\int_{\mathbb{T}}[\varphi_n^{R,*} (e^{i\theta})^\dagger{d\sigma} \varphi_n^{R,*} (e^{i
\theta})]\beta_n^{1/2}\right)=\log\left(\beta_n^{1/2}{\bf 1}\beta_n^{1/2}\right)=\log\beta_n\,,
\end{multline}
in view of the orthonormality  of the polynomials $\varphi^R_n$.
Next, we have
\begin{multline*}
\text{tr}\log\left(\beta_n^{1/2}[\varphi_n^{R,*} (e^{i\theta})^\dagger \sigma^{\prime} \varphi_n^{R,*} (e^{i
\theta})]\beta_n^{1/2}\right)=\log\det\left(\beta_n^{1/2}[\varphi_n^{R,*} (e^{i\theta})^\dagger \sigma^{\prime} \varphi_n^{R,*} (e^{i
\theta})]\beta_n^{1/2}\right)\\=
\log\left[\det(\beta_n)\det([\varphi_n^{R,*} (e^{i\theta})^\dagger  \varphi_n^{R,*} (e^{i
\theta})])\det(\sigma^{\prime})\right]=\log\det(\beta_n)+\log\det([\varphi_n^{R,*} (e^{i\theta})^\dagger  \varphi_n^{R,*} (e^{i
\theta})])+\log\det(\sigma^{\prime})\\=
\text{tr}\log\beta_n+\text{tr}\log[\varphi_n^{R,*} (e^{i\theta})^\dagger  \varphi_n^{R,*} (e^{i
\theta})]+\text{tr}\log\sigma^{\prime}.
\end{multline*}
Integrating the above equality and taking into account
 \eqref{an89}, \eqref{ident65} and \eqref{importantineq}, we arrive at
\begin{multline*}
\text{tr}\log\beta_n\geq \int_{\mathbb{T}}\text{tr}\log
\left(\beta_n^{1/2}[\varphi_n^{R,*} (e^{i\theta})^\dagger \sigma^{\prime}\varphi_n^{R,*} (e^{i
\theta})]\beta_n^{1/2}\right)
\frac{d\theta}{2\pi}\\=
\text{tr}\log\beta_n+\int_{\mathbb{T}}\text{tr}\log
\left([\varphi_n^{R,*} (e^{i\theta})^\dagger \varphi_n^{R,*} (e^{i
\theta})]\right)\frac{d\theta}{2\pi}+\int_{\mathbb{T}}\text{tr}\log\sigma^{\prime}\frac{d\theta}{2\pi}=
\int_{\mathbb{T}}\text{tr}\log\sigma^{\prime}\frac{d\theta}{2\pi}\,.
\end{multline*}
It remains to apply Lemma~\ref{detid0}.
\end{proof}
\begin{corollary}\label{leftineqb} If $\sigma$ is a Szeg\H{o} measure, then
\begin{equation}\label{firstin56}
\int_{\mathbb{T}}\tr\log\sigma^{\prime}\frac{d\theta}{2\pi}\leq\inf_n\text{tr}\log\beta_n\leq-\sup_n\log\Vert\beta_n^{-1}\Vert \leq 0,
\end{equation}
in particular, $\sup_n\Vert\beta_n^{-1}\Vert<+\infty$.
\end{corollary}
\begin{proof} Since $\beta_n$ satisfies \eqref{bettaop9}, all its  eigenvalues $\lambda_k$, $1\leq k\leq \ell$, must lie in the interval $(0,1]$. In addition,
\begin{equation*}
\Vert\beta_n^{-1}\Vert=\max_{1\leq k\leq \ell}\lambda_k^{-1}.
\end{equation*}
By \eqref{HalfEq} and Lemma~\ref{triv89}
\begin{equation*}
-\infty<\int_{\mathbb{T}}\tr\log\sigma^{\prime}\frac{d\theta}{2\pi}\leq\log\det\left(\beta_n\right)=\tr\log\beta_n\leq 0,
\end{equation*}
 implying that
\begin{equation*}
\log\Vert\beta_n^{-1}\Vert=\max_k\log\lambda_k^{-1}<\sum_{k=1}^{\ell}\log\lambda_k^{-1}=\tr\log\beta_n^{-1}\leq-\int_{\mathbb{T}}\tr\log\sigma^{\prime}\frac{d\theta}{2\pi}<+\infty.
\end{equation*}
\end{proof}

By compactness of closed balls in the finite-dimensional space ${\mathcal M}_\ell$, a
bounded  sequence of matrices has a limit point. It follows that if $\{\beta_n^{-1}\}_{n\geq 0}$ is uniformly bounded, then any of its limit points $\beta^{-1}$ in ${\mathcal M}_\ell$ satisfies
\begin{equation}\label{boundbeta}
\Vert\beta^{-1}\Vert\leq \sup_n\Vert\beta_n^{-1}\Vert<+\infty.
\end{equation}
We denote by $C(\mathbb{T}, {\mathcal M}_\ell^{+})$ the set of all continuous matrix functions on $\mathbb{T}$ with values in ${\mathcal M}_\ell^{+}$,
by $L^2(\mathbb{T}, {\mathcal M}_\ell^{+})$ the set of all square-integrable matrix functions
on $\mathbb{T}$ with values in ${\mathcal M}_\ell^{+}$, and by $M(\mathbb{T}, {\mathcal M}_\ell^{+})$ the set of all finite Borel measures with values in ${\mathcal M}_\ell^{+}$.

 Let $\mu$ be a finite Borel measure with values in ${\mathcal M}_\ell^{+}$.
Suppose that, for any open arc $I\subset\mathbb{T}$ whose endpoints do not carry point masses of $\sigma$, the inequality
\begin{equation*}
\mu(I)\preceq\sigma(I)
\end{equation*}
holds. Then we write $d\mu\preceq d\sigma$.
\begin{theorem}\label{main295} Let $\sigma\in\textsf{P}_{\ell}(\mathbb{T})$ satisfy $\sup_n\Vert\beta_n^{-1}\Vert<+\infty$, with $\beta_n$ defined as above, and let $\{f_n\}_{n\geq 0}$ be a sequence in $C(\mathbb{T}, {\mathcal M}_\ell^{+})$ such that $f_n(e^{i\theta})\succ \mathbf{0}$ on $\mathbb{T}$ and let
\begin{align}
\int_{\mathbb{T}}f_n&\frac{d\theta}{2\pi}\preceq \mathbf{1};\label{formula1}\\
*{-}\lim_nf_n&\frac{d\theta}{2\pi}\preceq d\sigma;\label{formula2}\\
\log\beta_n\preceq \int_{\mathbb{T}}&\log f_n\frac{d\theta}{2\pi}.\label{formula3}
\end{align}
Then
\begin{equation}\label{matrixeq}
\lim_n\log\beta_n=\int_{\mathbb{T}}\log\sigma^{\,\prime}\frac{d\theta}{2\pi}
\end{equation}
in ${\mathcal M}_\ell$.
\end{theorem}
\begin{proof}By \eqref{bettaop9} and \eqref{boundbeta}, the sequence of negative operators $\log\beta_n$ is uniformly bounded. Suppose that $\log\beta$ is a limit point of this sequence of matrices in ${\mathcal M}_\ell$. Then there is an infinite subset $\Lambda$ of $\mathbb N$ such that
\begin{equation}\label{limitbetan}
\lim_{n\in\Lambda}\log\beta_n=\log\beta.
\end{equation}

Let
\[
\log^{+} x=\max(\log x, 0)\;,\; \log^{-} x= \log^{+} x-\log x\;.
\]
Then $\log^{+} x\leq x$ for every $x>0$. The Spectral Theorem applied to a (strictly)
 positive operator $A$ yields
\begin{equation}\label{logAa2}
\log^{+}(A)\preceq A .
\end{equation}
We apply \eqref{logAa2} to $A:=f_n(e^{i\theta})$ pointwise in $\theta$ and
obtain the operator inequality
\begin{equation}\label{oppointinlog1}
\log^{+}(f_n(e^{i\theta}))\preceq f_n(e^{i\theta}).
\end{equation}
Integrating \eqref{oppointinlog1} and taking into account \eqref{formula1}, we obtain
\begin{equation}\label{IneqJen3.61}
\int_{\mathbb{T}}\log^{+}(f_n(e^{i\theta}))\frac{d\theta}{2\pi}\preceq \mathbf{1}.
\end{equation}
Observing that $\log=\log^{+}-\log^{-}$ and using \eqref{IneqJen3.61} and \eqref{formula3},
we see that
\begin{equation}\label{eqRT3.71}
\int_{\mathbb{T}}\log^{-}(f_n(e^{i\theta}))\frac{d\theta}{2\pi}\preceq \mathbf{1}+\log\beta_n^{-1}.
\end{equation}
Let
\begin{equation}
d\nu_n^{+}\eqbd \log^{+}(f_n(e^{i\theta}))\frac{d\theta}{2\pi},\qquad d\nu_n^{-}\eqbd \log^{-}(f_n(e^{i\theta}))\frac{d\theta}{2\pi}.
\end{equation}
Since $\{\beta_n^{-1}\}_{n\geq 0}$ is bounded, (\ref{eqRT3.71})
implies that $\{\nu_n^{-}\}_{n\geq 0}$ has a $*$-weak limit
point $\nu^-\in M(\mathbb{T}, {\mathcal M}_\ell^{+})$:
\begin{equation}\label{limmuminus5221}
d\nu^-=*{-}\lim_{n\in\Lambda^{'}} d\nu_n^-=(\nu^-)'\frac{d\theta}{2\pi}+d\nu^-_s\quad
\text{\rm for some } \; \Lambda^{'}\subset\Lambda,
\end{equation}
where $d\nu^-_s$ is the singular part of $d\nu^-$ (it may include the discrete part as well),
$(\nu^-)'=d\nu^-/(\frac{d\theta}{2\pi})$.
It follows from the inequality $(\log^{+} x)^2\leq x$ and \eqref{formula1} that
\begin{equation}\label{IneqJen3.6h1}
\int_{\mathbb{T}}\left(\log^{+}(f_n(e^{i\theta})\right)^2\,\frac{d\theta}{2\pi}
\preceq{\bf 1}\;,
\end{equation}
By (\ref{IneqJen3.6h1}), the function $d\nu_n^+/(\frac{d\theta}{2\pi})$ is in the unit ball of
$L^2(\mathbb{T}, {\mathcal M}_\ell)$, which is compact in the weak topology of
$L^2(\mathbb{T}, {\mathcal M}_\ell)$, see Theorem \ref{Ltwoth}. It follows that 
any $*$-limit point $\omega$ of
$\{\nu_n^+\}_{n\geq 0}$ in $M(\mathbb{T}, {\mathcal M}_\ell)$ is absolutely
continuous with respect to the Lebesgue measure and, moreover, belongs to $L^2(\mathbb{T}, {\mathcal M}_\ell^{+})$. Then
there exist a subset $\Lambda^{''}\subset\Lambda^{'}$ and some $\omega^{\prime}$ in
the unit ball of $L^2(\mathbb{T}, {\mathcal M}_\ell^{+})$ such that
\begin{equation}\label{mainpaper523}
\begin{gathered}
d\nu^+ \eqbd *{-}\lim_{n\in\Lambda^{''}}d\nu_n^+=\omega^{\prime}\frac{d\theta}{2\pi}
\;,\qquad *{-}\lim_{n\in\Lambda^{''}}d\nu^{-}_n=d\nu^{-}\;,\\
d\nu \eqbd d\nu^+-d\nu^-=(\omega^{\prime}-(\nu^-)') \frac{d\theta}{2\pi}-d\nu^-_s,
\end{gathered}
\end{equation}
see \eqref{limmuminus5221}. Let $I$ be an open arc on $\mathbb{T}$ such that its endpoints do
not carry point masses of $d\nu^-_s$ or $d\sigma_s$. By matrix Jensen's
inequality  from Proposition \ref{prJensen}, we get
\begin{equation}\label{JenSz3.101}
\frac{1}{|I|}\int_{I}\log(f_n(e^{i\theta}))\frac{d\theta}{2\pi}\preceq
\log\left\{\frac{1}{|I|}\int_I f_n(e^{i\theta})\frac{d\theta}{2\pi}\right\}\;.
\end{equation}
Applying Helly's  Theorem \ref{Helley1} separately to
$\{\nu_n^+\}_{n\in\Lambda''}$ and to
$\{\nu_n^-\}_{n\in\Lambda''}$, we obtain
\begin{equation}\label{mainpap5261}
\lim_{n\in\Lambda''}\frac{1}{|I|}\int_{I}\log(f_n(e^{i\theta}))\frac{d\theta}{2\pi}=\frac{\nu(I)}{|I|}\;.
\end{equation}
Applying Helly's Theorem \ref{Helley1},
we derive from~\eqref{formula2} the inequality
\begin{equation}\label{mainpap5271}
\lim_{n}\frac{1}{|I|}\int_If_n(e^{i\theta})\frac{d\theta}{2\pi}\preceq \frac{\sigma(I)}{|I|}\;.
\end{equation}
A substitution of (\ref{mainpap5261}) and (\ref{mainpap5271}) into
(\ref{JenSz3.101}) results in the inequality
\[
\frac{\nu(I)}{|I|}\preceq \log\left(\frac{\sigma(I)}{|I|}\right)\;
\]
(here we use the operator continuity of the logarithm, see Proposition \ref{OpCont}). 
It follows from Lebesgue's
theorem on differentiation and the operator continuity of the logarithm that
\begin{equation}\label{LebDiff3.111}
\nu^{\prime}\preceq \log(\sigma^{\prime})
\end{equation}
almost everywhere on $\mathbb{T}$. In view of (\ref{LebDiff3.111})
and (\ref{formula3}), we obtain
\begin{equation}\label{MainPaper5291}
\log{\beta}+\nu^-_s(\mathbb{T})\preceq\int_{\mathbb{T}}d\nu+\nu^-_s(\mathbb{T})=
\int_{\mathbb{T}}\nu^{\prime}\frac{d\theta}{2\pi}
\preceq\int_{\mathbb{T}}\log\sigma^{\prime}\frac{d\theta}{2\pi}\;.
\end{equation}
Combining~\eqref{HalfEq} with (\ref{MainPaper5291}), we see that
\begin{equation}\label{logsigmaform85941}
\int_{\mathbb{T}}\tr\log\sigma^{\prime}\frac{d\theta}{2\pi}=\tr\log\beta.
\end{equation}
and $\tr\nu_s(\mathbb{T})=0$, so $\nu^-_s=0$ by the nonnegativity of the measure $\nu^-_s$. 
It follows that
\begin{equation*}\log{\beta}\preceq\int_{\mathbb{T}}\log\sigma^{\prime}\frac{d\theta}{2\pi}.
\end{equation*}
Since the traces of the operators on both sides are equal by \eqref{logsigmaform85941},
we  invoke Lemma~\ref{triv88} and conclude that
\begin{equation*}
\log{\beta}=\int_{\mathbb{T}}\log\sigma^{\prime}\frac{d\theta}{2\pi}.
\end{equation*}
Since $\log\beta$ is an arbitrary limit point of $\{\log\beta_n\}_{n\geq 0}$,
we obtain \eqref{matrixeq}.
\end{proof}
\begin{theorem}\label{secondth0}Let $\sigma\in\textsf{P}_{\ell}(\mathbb{T})$ satisfy $\sup_n\Vert\beta_n^{-1}\Vert<+\infty$. Then
\begin{equation*}
\lim_n\log\beta_n=\int_{\mathbb{T}}\log\sigma^{\,\prime}\frac{d\theta}{2\pi} .
\end{equation*}
\end{theorem}
\begin{proof}  Set
$$
f_n(e^{i\theta})=[\varphi_n^{R,*} (e^{i\theta})^\dagger \varphi_n^{R,*} (e^{i
\theta})]^{-1}
$$
in Theorem \ref{main295}. Then \eqref{formula1} and \eqref{formula2} follow from Theorem \ref{BSconv}. Finally, \eqref{formula3} follows from \eqref{an89}.
\end{proof}

\begin{theorem}[{\cite[Theorem 18]{DGK:78}}]\label{szego_measure} 
For any $\sigma\in\textsf{P}_{\ell}(\mathbb{T})$,
\begin{equation}\label{Szlogcond3.4}
\log\prod_{k=0}^{\infty}\det(1-\alpha_k^\dagger \alpha_k)=
\int_{\mathbb{T}}\tr\log\sigma^{\prime}\frac{d\theta}{2\pi}\;.
\end{equation}
\end{theorem}
\begin{proof} Since $\det(1-\alpha_k^\dagger\alpha_k)<1$ for all $k$, the sum of the series
\begin{equation*}
\sum_{k=0}^{\infty}\log\det(1-\alpha_k^\dagger\alpha_k)
\end{equation*}
with negative terms satisfies
\begin{equation}\label{pofght68re}
\sum_{k=0}^{\infty}\log\det(1-\alpha_k^\dagger\alpha_k)\geq\int_{\mathbb{T}}\tr\log\sigma^{\prime}\frac{d\theta}{2\pi}
\end{equation}
by Theorem \ref{firstineq}. We have two cases. If the series on the left-hand side of \eqref{pofght68re} diverges, then $\sigma$ is not a Szeg\"o measure and both sides of \eqref{Szlogcond3.4} equal $-\infty$. If the series on the left-hand side of \eqref{pofght68re} converges, then
\begin{equation*}
\lim_k\log\det(1-\alpha_k^\dagger\alpha_k)=\lim_k\text{tr}\log(1-\alpha_k^\dagger\alpha_k)=0.
\end{equation*}
Since the spectral norm $\|\cdot \|$ is the largest eigenvalue of a positive self-adjoint matrix,
it follows that
\begin{equation*}
\lim_k\Vert\alpha_k^\dagger\alpha_k\Vert=0.
\end{equation*}
Since $-x\geq\log(1-x)$ for $0<x<1$, we see that
\begin{equation*}
-\Vert\alpha_k^\dagger\alpha_k\Vert\geq \log(1-\Vert\alpha_k^\dagger\alpha_k\Vert)\geq \text{tr}\log(1-\alpha_k^\dagger\alpha_k),
\end{equation*}
implying that
\begin{equation*}
\sum_{k=0}^{\infty}\Vert\alpha_k^\dagger\alpha_k\Vert<+\infty.
\end{equation*}
Lemma \ref{detid0} implies  that the $\Vert\beta_n^{-1}\Vert$ are bounded. An
application of Theorem \ref{secondth0} now completes the proof.
\end{proof}
\begin{corollary}[{\cite[Theorem 19]{DGK:78}}]\label{parametersth}
 A measure $\sigma\in\textsf{P}_{\ell}(\mathbb{T})$ is a Szeg\"o measure if and only if
\begin{equation*}
\sum_{k=0}^{\infty}\Vert\alpha_k^\dagger\alpha_k\Vert<+\infty.
\end{equation*}
\end{corollary}

\noindent
One direction of this corollary was already proved in Theorem~\ref{szego_measure}; the other
direction can be obtained analogously to~\cite{DGK:78}

\begin{corollary}\label{colSz3.1} Let $\sigma$ be a Szeg\"o measure
and let $\{\varphi_n\}_{n\geq 0}$ be the orthogonal polynomials in
$L^2(d\sigma)$. Then
\begin{equation}\label{weakconvSz3.14}
*{-}\lim_n d\mu_n= *{-}\lim_n\log\left([\varphi_n^{R,*} (e^{i\theta})^\dagger \varphi_n^{R,*} (e^{i
\theta})]^{-1}\right)\frac{d\theta}{2\pi}=\log(\sigma^{\prime})\frac{d\theta}{2\pi}
\end{equation}
in the weak topology of $M(\mathbb{T, {\mathcal M}_\ell})$.
\end{corollary}
\begin{proof} 
Apply the proof of Theorem~\ref{main295} to the measures
\begin{eqnarray*}
 d\nu_n^+ \eqbd d\mu_n^+ \eqbd 
\log^+\left([\varphi_n^{R,*} (e^{i\theta})^\dagger \varphi_n^{R,*} 
(e^{i\theta})]^{-1}\right) \frac{d\theta}{2\pi},  \\
 d\nu_n^- \eqbd d\mu_n^- \eqbd 
\log^-\left([\varphi_n^{R,*} (e^{i\theta})^\dagger \varphi_n^{R,*} 
(e^{i\theta})]^{-1}\right) \frac{d\theta}{2\pi}.
\end{eqnarray*}
Taking \eqref{MainPaper5291} into account, we obtain
\begin{equation}\label{mainpape530}
\nu^{\prime}=\log\sigma^{\prime}\quad \text{  a.e. on  } \mathbb{T}\;.
\end{equation}
The substitution of (\ref{mainpape530}) into the last formula of
(\ref{mainpaper523}) results in
\[
d\nu=\log\sigma^{\prime}\frac{d\theta}{2\pi}=(\omega^{\prime}-(\nu^-)')\frac{d\theta}{2\pi}\;.
\]
Since $\omega$ in the proof of Theorem~\ref{main295} was an arbitrary $*$-limit point of
$\{\nu_n^{+}\}_{n\in\Lambda}$, this implies that
$*{-}\lim_{n\in\Lambda'}d\nu_n^+=\omega^{\prime}\frac{d\theta}{2\pi}$. Since $\nu^-$ was an
arbitrary $*$-limit point of $\{\nu_n^{-}\}_{n\in\Lambda}$, we
conclude that $*{-}\lim_{n}d\mu_n=\log\sigma^{\prime}\frac{d\theta}{2\pi}$.
\end{proof}

\section{The Helson-Lowdenslager Theorem}

Since $\mathbf{\Phi}_n^{R,*}$ is left orthogonal to $z\mathbf{1},\ldots,z^n\mathbf{1}$ (see \cite[Lemma 3.2]{DPS:08}), it is also left orthogonal to any linear combination $p$ of these matrix functions with the coefficients in ${\mathcal M}_\ell$. Take any such combination $p$. Then
\begin{multline*}
\ang{\mathbf{\Phi}_n^{R,*}-p,\mathbf{\Phi}_n^{R,*}-p}_L=
\ang{\mathbf{\Phi}_n^{R,*},\mathbf{\Phi}_n^{R,*}}_L+\ang{p,p}_L-\ang{\mathbf{\Phi}_n^{R,*},p}_L
-\ang{\mathbf{\Phi}_n^{R,*},p}_L^\dagger=
\ang{\mathbf{\Phi}_n^{R,*},\mathbf{\Phi}_n^{R,*}}_L+\ang{p,p}_L.
\end{multline*}
Since every polynomial $\mathbf{Q}$ satisfying $\mathbf{Q}(0)=\mathbf{1}$ is of the form $\mathbf{Q}=\mathbf{\Phi}_n^{R,*}-p$ we obtain the matrix inequality
\begin{equation}\label{varprincip}
\ang{\mathbf{\Phi}_n^{R,*},\mathbf{\Phi}_n^{R,*}}_L\preceq \ang{\mathbf{Q},\mathbf{Q}}_L.
\end{equation}
These facts are also derived for the real line in~\cite[Formula (2.10)]{DPS:08}.

It is therefore natural to call the square root of the positive matrix in the left-hand side of \eqref{varprincip} the {\em left operator distance\/} from $\mathbf{1}$ to $z\mathcal{P}_{n-1}$.
Consequently, the usual distance in the left Hilbert space is equal to
\begin{equation}\label{distform}
\text{dist}_L\left(\textbf{1},z\mathcal{P}_{n-1}\right)^2=\text{tr}\left(\ang{\mathbf{\Phi}_n^{R,*},\mathbf{\Phi}_n^{R,*}}_L\right).
\end{equation}
One easily verifies (see also \cite[Lemma 3.1]{DPS:08}) that
\begin{equation}\label{eval745}
\begin{array}{lcl} \ang{\mathbf{\Phi}_n^{R,*},\mathbf{\Phi}_n^{R,*}}_L & = & \ang{\mathbf{\Phi}_n^{R},\mathbf{\Phi}_n^{R}}_R^\dagger=\ang{\mathbf{\Phi}_n^{R},\mathbf{\Phi}_n^{R}}_R 
= (\kappa_n^{R})^{-\dagger}\ang{{\varphi}_n^{R},{\varphi}_n^{R}}_R(\kappa_n^{R})^{-1}
\\[2mm]
&=& (\kappa_n^{R})^{-\dagger}\mathbf{1}
(\kappa_n^{R})^{-1}=(\kappa_n^{R})^{-\dagger}(\kappa_n^{R})^{-1}.
\end{array}
\end{equation}
It follows from \eqref{kappaformula} that
\begin{equation}\label{distformula78}
((\kappa_n^{R})^{-\dagger}(\kappa_n^{R})^{-1})^{1/2} =\rho_{n-1}^R\cdots\rho_{0}^R=({\bf 1} -\alpha_{n-1} \alpha_{n-1}^\dagger)^{1/2}\cdots
({\bf 1} -\alpha_{0} \alpha_{0}^\dagger)^{1/2}
\end{equation}
is the left matrix distance from $\mathbf{1}$ to $z\mathcal{P}_{n-1}$.
\begin{corollary}The identity polynomial $\mathbf{1}$ is in the left closure of the sets of matrix polynomials $z\mathcal{P}_{n-1}$ if and only if
\begin{equation*}
\exp\int_{\mathbb{T}}\log\sigma^{\,\prime}\frac{d\theta}{2\pi}=\mathbf{0}.
\end{equation*}
\end{corollary}
The distance formula \eqref{distformula78} is useful if the parameters $\{\alpha_k\}_{k\geq0}$ of $\sigma$ are known. If this is not the case, then one can apply an estimate for $(\kappa_n^{R})^{-1}$ from below which was obtained by Helson-Lowdenslager in \cite{HelsonLowdenslager}.

Now we are in a position to prove the main result of \cite{HelsonLowdenslager}.
\begin{theorem}[\cite{HelsonLowdenslager}]\label{HL} For every $\sigma\in\textsf{P}_{\ell}(\mathbb{T})$
\begin{equation}\label{HLeq}
\exp\int_{\mathbb{T}}\frac{1}{\ell}\text{tr}\log\sigma^{\,\prime}\,\frac{d\theta}{2\pi}=
\inf\limits_{A,P}\int_{\mathbb{T}}\frac{1}{\ell}\tr\left[(A+P)^\dagger
d\sigma(A+P)\right],
\end{equation}
where $A$ runs over all matrices with determinant one, and $P$ over
all trigonometric polynomials of the form
\[
P(e^{i\theta})=\sum_{k>0}A_ke^{ik\theta}.
\]
\end{theorem}
\begin{proof}

Combining Lemma \ref{HL-Lemma} with formula \eqref{distformula78}, we get
\begin{eqnarray*}
&& \inf_{A\in\mathcal{A}}
\frac{1}{\ell}\text{tr}
 A \left( (\kappa_n^R)^{-\dagger} 
 (\kappa_n^R)^{-1}   \right) A^\dagger  =
[\det((\kappa_n^{R})^{-\dagger} (\kappa_n^R)^{-1})  ]^{1/\ell} \\ && =
\exp\left\{\frac{1}{\ell}\sum_{j=0}^{n-1}\log\det(1-\alpha_j\alpha_j^\dagger)\right\}=
\exp\left\{\frac{1}{\ell}\text{tr}(\log\beta_n)\right\}.
\end{eqnarray*}
Combining this formula with \eqref{distform}, we obtain
\begin{eqnarray*}
\inf\limits_{A\in\mathcal{A},P\in z\mathcal{P}_{n-1}}\int_{\mathbb{T}}\frac{1}{\ell}\tr\left[(A+P)
d\sigma(A+P)^\dagger\right]=\inf\limits_{A\in\mathcal{A},P\in z\mathcal{P}_{n-1}}
\int_{\mathbb{T}}{1\over \ell}\tr A (\mathbf{1}+zA^{-1}P) d \sigma 
(\mathbf{1}+zA^{-1}P)^\dagger A^\dagger   \\
 = \inf_{\mathcal{A}} {1\over \ell} \tr A \left( (\kappa_n^R)^{-\dagger} 
 (\kappa_n^R)^{-1}   \right) A^\dagger  =
\exp\left\{\frac{1}{\ell}\int_{0}^{2\pi} \text{tr}\log([\varphi_n^{R,*} 
(e^{i\theta})^\dagger \varphi_n^{R,*} (e^{i\theta})]^{-1})\frac{d\theta}{2\pi}\right\}.
\end{eqnarray*}
Passing to the limit, we arrive at \eqref{HLeq}, which was initially proved via a different method in~\cite[Theorem~8]{HelsonLowdenslager}.
\end{proof}

\section*{Acknowledgments}
The research leading to these results was carried out at Technische
Universit\"at Berlin and has received funding from the
Alexander von Humboldt Foundation under the Sofja Kovalevskaja Prize Programme
and from the European Research Council under the European Union's Seventh
Framework Programme (FP7/2007-2013) / ERC grant agreement ${\rm n}^\circ$ 259173.

\end{document}